\documentclass[referee,pdflatex,sn-mathphys-ay]{sn-jnl}

\usepackage{graphicx}%
\usepackage{multirow}%
\usepackage{amsmath,amssymb,amsfonts}%
\usepackage{amsthm}%
\usepackage{mathrsfs}%
\usepackage[title]{appendix}%
\usepackage{xcolor}%
\usepackage{textcomp}%
\usepackage{manyfoot}%
\usepackage{booktabs}%
\usepackage{algorithm}%
\usepackage{algorithmicx}%
\usepackage{algpseudocode}%
\usepackage{listings}%

\usepackage{dsfont}
\usepackage{latexsym}
\usepackage{amssymb}
\usepackage{amsmath}
\usepackage{amsthm}
\usepackage{paralist}
\usepackage{upgreek}
\usepackage{soul}
\usepackage{subcaption}
\usepackage{caption}
\usepackage{rotating}
\usepackage{tikz}
\usetikzlibrary{positioning}
\usepackage{blkarray}
\usepackage{mathtools}
\usepackage{pgfplots}
\usepackage{fancyhdr}
\usepackage{color}
\usepackage{tocbibind}
\usepackage{fancyhdr}
\usepackage{epsfig}
\usepackage{xcolor}
\usepackage{lineno,hyperref}
\hypersetup{colorlinks = true, linkcolor = blue, anchorcolor =red, citecolor = blue,filecolor = red, urlcolor = red, pdfauthor=author}

\usepackage{rotating}
\usepackage{url}
\usepackage{lscape}
\usepackage{float}
\usepackage{booktabs}\let\cline\cmidrule

\DeclarePairedDelimiter\abs{\lvert}{\rvert}

\let\oldabs\abs
\def\abs{\@ifstar{\oldabs}{\oldabs*}}

\DeclareMathAlphabet\mathbfcal{OMS}{cmsy}{b}{n}

\theoremstyle{thmstyleone}%
\newtheorem{theorem}{Theorem}
\newtheorem{proposition}[theorem]{Proposition}%

\theoremstyle{thmstyletwo}%
\newtheorem{example}{Example}%

\theoremstyle{thmstylethree}%

\begin{document}

\title[On \texorpdfstring{$\lambda$}{}-Cent-Dians and Generalized-Center for Network Design]{On \texorpdfstring{$\lambda$}{}-Cent-Dians and Generalized-Center for Network Design: Formulations and Algorithms}


\author[1,2]{\fnm{V\'ictor} \sur{Bucarey}}\email{victor.bucarey@uoh.cl}
\equalcont{These authors contributed equally to this work.}

\author*[3]{\fnm{Natividad} \sur{Gonz\'alez-Blanco}}\email{ngonzalez@uloyola.es}

\author[4,5]{\fnm{Martine} \sur{Labb\'e}}\email{mlabbe@ulb.ac.be}
\equalcont{These authors contributed equally to this work.}

\author[6,7]{\fnm{Juan A.} \sur{Mesa}}\email{jmesa@us.es}
\equalcont{These authors contributed equally to this work.}

\affil[1]{\orgdiv{Institute of Engineering Sciences}, \orgname{Universidad de O'Higgins}, \orgaddress{\city{Rancagua}, \country{Chile}}}

\affil[2]{\orgdiv{Instituto Sistemas Complejos de Ingenier\'ia (ISCI)}, \orgaddress{\city{Santiago Centro}, \country{Chile}}}

\affil[3]{\orgname{Universidad Loyola Andalucía}, \orgdiv{Departamento de M\'etodos Cuantitativos},  \orgaddress{\city{Dos Hermanas}, \country{Spain}}}

\affil[4]{\orgdiv{D\'epartement d'Informatique}, \orgname{Universit\'e Libre de Bruxelles}, \orgaddress{\city{Brussels}, \country{Belgium}}}

\affil[5]{\orgdiv{Inria Lille-Nord Europe}, \orgaddress{ \city{Villeneuve d'Ascq}, \country{France}}}

\affil[6]{\orgdiv{Departamento de Matem\'atica Aplicada II}, \orgname{Universidad de Sevilla}, \orgaddress{\city{Sevilla}, \country{Spain}}}

\affil[7]{\orgdiv{Instituto de Matem\'aticas de la Universidad de Sevilla}, \orgname{Universidad de Sevilla}, \orgaddress{\city{Sevilla}, \country{Spain}}}


\abstract{In this paper, we study the $\lambda$-centdian problem in the domain of Network Design. The focus is on designing a sub-network within a given underlying network while adhering to a budget constraint. This sub-network is intended to efficiently serve a collection of origin/destination demand pairs. We extend the work presented in \cite{bucarey2024on}, providing an algorithmic perspective on the generalized $\lambda$-centdian problem.

In particular, we provide a mathematical formulation for $\lambda\geq 0$ and discuss the bilevel structure of this problem for $\lambda>1$. Furthermore, we describe a procedure to obtain a complete parametrization of the Pareto-optimality set based on solving two mixed integer linear formulations by introducing the concept of maximum $\lambda$-cent-dian. We evaluate the quality of the different solution concepts using some inequality measures. Finally, for $\lambda\in[0,1]$, we study the implementation of a Benders decomposition method to solve it at scale.}


\maketitle

\section{Introduction}\label{sec1}

The \textit{center} and \textit{median} problems in graphs and Euclidean spaces were central to Location Science in the 1950s and 60s. Median problems focus on efficiency, while center problems aim to maximize equity. The median problem minimizes the sum of weighted distances from demand points to the nearest facility, while the center problem minimizes the largest such distance.

The median problem suits cost minimization, while the center problem is ideal for emergency services. However, separately, they do not balance efficiency and equity.

In \cite{halpern1976location}, the term $\lambda$ cent-dian was introduced for location problems that aimed to minimize a linear combination of center and median objectives, $F_c$ and $F_m$, represented as $H_{\lambda} = \lambda F_c + (1-\lambda)F_m$. In \cite{halpern1978finding}, it was shown that the $\lambda$ cent-dian of a graph lies on a path between the center and the median. \cite{hansen1991median} proposed an $O(|N||E|\log(|N||E|))$ algorithm to find all $\lambda$-cent-dian points and introduced the \textit{generalized-center}, the minimizer of the difference between the center and median. As $\lambda \rightarrow \infty$, the ratio $\frac{H_{\lambda}}{\lambda}$ tends to the generalized-center. This objective favors equity between O/D pairs, but may lead to inefficient solutions, as noted by \cite{ogryczak1997centdians}. An axiomatic approach to the $\lambda$-cent-dian criterion was provided in \cite{carrizosa1994axiomatic}.

The $\lambda$-cent-dian objective has also been applied to extensive facility location problems, where facilities are too large to be represented by isolated points. Examples of extensive facilities include paths, cycles, or trees on graphs, and straight lines, circles, or hyperplanes in Euclidean spaces (see \cite{diaz2004continuous}, \cite{mesa1996review}, \cite{puerto2009extensions}, \cite{schmidt2014location}).

The $\lambda$-cent-dian criterion balances two conflicting objectives. We are the first to formalize the $\lambda$-cent-dian network design problem, focusing on optimal sub-networks and O/D pair demand instead of single-point facilities.

Network Design problems are applied in telecommunications, transportation, and other fields. In some cases, demand is represented by O/D pairs, generating flows that the network must manage. Multiple networks may compete to capture demand, as seen in mobile telephony and urban mobility, where different transportation modes compete for commuters' preferences.

Most Network Design problems focus on cost or profit, often serving as surrogates for median problems with weighted distance summation. Some scenarios, such as daily commute in metropolitan areas or electricity distribution in rural regions, require considering proximity, where the center objective must be combined with the median objective due to concerns about time or power loss (see \cite{gokbayrak2022two}).

In this article, we exploit the structure of the problem when $\lambda\in [0,1]$ by exploring a Benders Decomposition approach. This scheme has been extensively studied for Network Design problems (see \cite{magnanti1986tailoring, fortz2009improved, marin2009urban, botton2013benders}). We also extend the strong cuts introduced by \cite{conforti2019facet2} to this setting. These cuts have shown competitive results in many applications in Facility Location and Network Design Problems (see \cite{cordeau2019benders, bucarey2022benders}).

The contributions of this work are the following.
\begin{itemize}
		\item We provide a mathematical formulation for the $\lambda$-cent-dian problem with $\lambda\geq 0$, which has been presented in \cite{bucarey2024on}. Furthermore, we address the bilevel structure when $\lambda>1$. We then convert the bilevel problem into a valid single-level problem formulation. Additionally, we integrate the criterion for efficiency by imposing a lower bound for the median value of the solution network.
	\item We outline a method to give a complete parametrization of the Pareto-optimality set based on solving two linear formulations.
    \item We evaluate and discuss the quality of the different solution concepts considered using some inequality measures. This evaluation incorporates the generalized-center concept, which has been explored in the existing literature. As you can see, by considering the criterion for efficiency referred to above, this concept is not always dominated. It depends on the measures used. 
	\item Finally, in order to facilitate the solution of large instances, we tested a Benders decomposition approach to solve the problem for the case $\lambda\in[0,1]$. We tackle the problem using a Benders decomposition implementation and the ideas exposed in \cite{conforti2019facet2}. Our computational experiments show that our Benders implementation is competitive with exact and non-exact methods existing in the literature and even compared with the exact method of Benders decomposition existing in \texttt{CPLEX}.
\end{itemize}

The structure of the paper is as follows. In Section \ref{sec:problem_definition}, we present the problem. In Section \ref{sec:problem_formulation}, we propose a mathematical formulation for $\lambda\geq 0$. In addition, we discuss the bilevel structure when $\lambda>1$, and adjust the formulation for this case. Then, we describe a procedure to give a complete parametrization of the Pareto-optimality set based on solving two linear formulations. Then, in Section \ref{sec:delta_results}, we investigate the quality of the solution concepts developed in this work. We develop a Benders decomposition approach to solve the problem for the case $\lambda\in[0,1]$ along with efficient tools to find facet-defining cuts provided in \cite{conforti2019facet2}, shown in Section \ref{sec:alg_discussion}. Finally, our conclusions are presented in Section \ref{sec:conclusions}.

It is important to note that in this paper we will refer to some propositions, examples, and results obtained in \cite{bucarey2024on} since, in this part, we are focused on the solution to the problems presented in it.

\section{Problems definition}\label{sec:problem_definition}

In this section, we briefly describe the setting and solution concepts that will be modeled throughout the paper, which were first introduced in \cite{bucarey2024on}.

\subsection{Setting description}
In this paper, we address a network design problem based on an undirected graph $\mathcal{N} = (N, E)$, with node costs $b_i \geq 0$ for $i \in N$ and edge costs $c_e \geq 0$ for $e \in E$. Each edge $e = \{i, j\} \in E$ generates two arcs: $a = (i, j)$ and $\hat{a} = (j, i)$, forming the set of arcs $A$. Arc lengths $d_a \geq 0$ may represent times to traverse, generalized costs, or other parameters. For each node $i \in N$, $\delta(i)$ is the set of incident edges, while $\delta^+(i)$ and $\delta^-(i)$ denote outgoing and incoming arcs, respectively.

The mobility patterns are represented by a set $W \subset N \times N$ of origin-destination (O/D) pairs with associated demands $g^w > 0$, where $w = (w^s, w^t) \in W$ and $w^s \neq w^t$. The total demand is $G = \sum_{w \in W} g^w$. Each pair also has a private utility $u^w > 0$, representing an existing mode of transport. An O/D pair is served/covered by the new network if it offers a path with a length not greater than $u^w$. For each $w \in W$, the subgraph $\mathcal{N}^w = (N^w, E^w)$ contains all the paths in $\mathcal{N}$ with lengths within $u^w$.

We consider subgraphs $\mathcal{S} = (N_{\mathcal{S}}, E_{\mathcal{S}})$ constrained by a budget $\alpha C_{total}$, where $C_{total}$ is the total cost of building $\mathcal{N}$. A subgraph $\mathcal{S}$ is feasible if:
\begin{equation}
    \sum_{i \in N_{\mathcal{S}}} b_i + \sum_{e \in E_{\mathcal{S}}} c_e \leq \alpha \left(\sum_{i \in N} b_i + \sum_{e \in E} c_e\right) = \alpha\, C_{total}.
\end{equation}
We denote the set of all feasible subgraphs for a given $\alpha \in [0, 1]$ as $\mathbfcal{N}^\alpha$.

For a subgraph $\mathcal{S}$ and O/D pair $w \in W$, $d_{\mathcal{S}}(w)$ represents the shortest path length between $w^s$ and $w^t$ in $\mathcal{S}$. If no connection exists, $d_{\mathcal{S}}(w) = +\infty$. The effective travel length of $w$ in $\mathcal{S}$ is
$\ell_{\mathcal{S}}(w) = \min\{d_{\mathcal{S}}(w), u^w\}$.

\subsection{Solution concepts}

For a given $\alpha\in (0,1)$, the (weighted) median is defined as a subgraph $\mathcal{S}_m$ in $\mathbfcal{N}^\alpha$ that minimizes the objective function
\begin{equation}
	F_m(\mathcal S) = \frac{1}{G} \sum_{w \in W} g^w\ell_{\mathcal{S}}(w). \label{median_obj}
\end{equation}

A subgraph $\mathcal{S}_c$ is called a center if it is a minimizer of the following objective function
\begin{equation}
	F_c(S) = \max_{w \in W} \ell_{\mathcal{S}}(w). \label{center_obj}
\end{equation}

\noindent Given that the center minimizes the maximum travel time, it could lead to \textit{inefficient} solutions. This inefficiency is produced by not considering feasible subgraphs in which the travel time decreases for some users while maintaining the center objective function value.

\tikzstyle{vertex}=[circle,fill=white, draw=black,minimum size=20pt,font=\footnotesize,inner sep=0pt]
\tikzstyle{vertex2}=[circle,minimum size=15pt,font=\scriptsize,inner sep=0pt]
\tikzstyle{selected vertex} = [vertex, fill=red!24]
\tikzstyle{edge} = [draw,thick,-]
\tikzstyle{edge2} = [draw,dashed,-]
\tikzstyle{weight} = [font=\scriptsize,inner sep=1pt]
\tikzstyle{selected edge} = [draw,line width=5pt,-,red!50]
\tikzstyle{ignored edge} = [draw,line width=5pt,-,black!20]

It has been highly motivating that there is a necessity of having a refined conceptual framework to consider solutions that are not dominated and capture the trade-off between both solution concepts (see Examples 1 and 2). Recognizing this need for balanced solutions, we considered a convex combination of both the center and the median objectives. In the realm of network design, for a given $\lambda \in[0,1]$, the $\lambda$-cent-dian is a subgraph that aims to minimize the following objective function:
\begin{equation}
	H_\lambda(S) = \lambda F_c(\mathcal S) + (1-\lambda)F_m(\mathcal S) = \lambda \max_{w \in W} \ell_{\mathcal{S}}(w) + (1-\lambda) \frac{1}{G} \sum_{w \in W} g^w\ell_{\mathcal{S}}(w), \label{eq:cent-dian}
\end{equation}

On the other hand, we deal with the concept of {\it generalized-center} in the network design area, which 
was introduced by \cite{hansen1991median} for the first time in the facility location area. This concept was formulated to reduce discrepancies in accessibility among users as much as possible. The {\it generalized-center} corresponds to a subgraph, $\mathcal S_{gc}$, which minimizes the disparity between the center and median objectives, denoted as $F_{gc}(\,\cdot\,)$. In fact, when $\lambda \rightarrow \infty$, the ratio $\frac{H_{\lambda}}{\lambda}$ tends to the generalized-center. 

As is shown in Example 3 of \cite{bucarey2024on}, the generalized-center can lead with unreasonable solutions from the point of view of the median value. To avoid such solution networks, the optimal solution network is restricted to the set of solution networks so-called as \textit{Pareto-optimal concerning the distances (of the shortest paths)}. We denote the set of subgraphs satisfying this definition of Pareto-optimality as $\mathbfcal{PO}^\alpha$. Thus, we redefined the generalized-center as the optimal solution to the problem
\begin{equation}\label{eq:GC}
	\min\{F_c(\mathcal{S})-F_m(\mathcal{S}): \mathcal{S}\in \mathbfcal{PO}^{\alpha}\}.
\end{equation}

Nevertheless, it is easy to see (Example 4) that even when the problem of the generalized-center is constrained to the set $\mathbfcal{PO}^{\alpha}$, there is potential for obtaining highly inefficient solution graphs as optimal solutions. 

There is a necessity to revisit the notion of Pareto-optimality. We perform this by extending the notion of Pareto-optimality to the bi-criteria setting: a subgraph $\mathcal{S}$ is Pareto-optimal if there no exists another subgraph $\mathcal{S'}$ in $\mathbfcal{N}^\alpha$ for which 
\begin{equation}
	F_m(S') \le F_m(S) \quad \mbox{ and } \quad F_c(S') \le F_c(S),
\end{equation}
\noindent with one of the inequalities being strict. We denote the set of subgraphs satisfying this property $\mathbfcal{PO}^\alpha_2$. Hence, we define the \textit{restricted-generalized-center} as an optimal solution to the problem 
\begin{equation}
	\min\{F_c(\mathcal{S})-F_m(\mathcal{S}): \mathcal{S}\in \mathbfcal{PO}^{\alpha}_2\}. \label{eq:restricted_general_center}
\end{equation}

Analogously to the restricted-generalized-center, we introduce the \textit{$\lambda$-restricted-cent-dian}. It is defined as the optimal solution to the problem:
\begin{equation}
	\min\{H_{\lambda}(\mathcal{S}): \mathcal{S}\in \mathbfcal{PO}^{\alpha}_2\}.
\end{equation}

It is easy to check that the generalized-center does not necessarily belong to the set $\mathbfcal{PO}^{\alpha}_2$ (see Example 4). In addition, in that work it is proved that for all $\lambda\geq 1$ the corresponding $\lambda$-restricted-cent-dian is always a center (Proposition 2). Hence, to belong to $\mathbfcal{PO}^{\alpha}_2$, the center must be unique, or it must be the center with the best value of $F_m(\,\cdot\,)$. Thus, the center belonging to $\mathbfcal{PO}^{\alpha}_2$ is an optimal solution of the following lexicographic problem
\begin{equation}\label{eq:lexmin}
	\mbox{lex}\min\{[F_c(\mathcal{S}),F_m(\mathcal{S})]:\,\,\mathcal{S}\in\mathbfcal{N}^{\alpha}\}.
\end{equation}

We refer to the optimal solution of \eqref{eq:lexmin} as a \textit{lexicographic cent-dian}. Then, we conclude that the $\lambda$-restricted-cent-dians for $\lambda\geq1$ are simply the centers with the best median value. 

So far, we have discussed Pareto-optimality concerning the center and median functions. We have only evaluated the solution concepts presented in this section from these two perspectives. As a result, the generalized center has been misjudged and overlooked. As we shall see later, the generalized center is not a dominated solution from the perspective of other inequality measures. However, to simply impose the degree of inefficiency of the solution network, we consider the following constraint. 
\begin{equation}\label{eq:delta_constraint}	F_m(\mathcal{S}^*)\leq(1+\Delta)\,F_m(\mathcal{S}_m), \quad\text{ with } \Delta\geq 0. 
\end{equation}

Finally, we justified that, for the purpose of identifying some compromise $\lambda$-cent-dians on a general network, we needed a solution concept different from the ones discussed so far. The Pareto-optimality set $\mathbfcal{PO}^{\alpha}_2$ can be completely parametrized through the solution of the lexicographic (two-level) problem
\begin{equation}\label{eq:lexmin2}
	\mbox{lex}\min\{[\Bar{H}_{\lambda}(\mathcal{S}),H_{\lambda}(\mathcal{S})]:\,\,\mathcal{S}\in\mathbfcal{N}^{\alpha}\},
\end{equation} 

where the weighted function $\Bar{H}_{\lambda}(\,\cdot\,)$ has the following expression
\begin{equation}\label{eq:chevfunc}
	\Bar{H}_{\lambda}(\mathcal{S}) = \max\{\lambda\,F_c(\mathcal{S}),(1-\lambda)\,F_m(\mathcal{S})\}.
\end{equation}

The second stage is necessary in the case of a non-unique optimal solution in the first stage. Thus, we call a subgraph $\mathcal{S}\in\mathbfcal{N}^{\alpha}$ a \textit{maximum $\lambda$-cent-dian} if it is an optimal solution of the lexicographic problem \eqref{eq:lexmin2}.

\section{Problem formulations}\label{sec:problem_formulation}

In this section, we present mixed-integer linear formulations for the $\lambda$-cent-dian problem as defined in equation~\eqref{eq:cent-dian}. Initially, we introduce a mixed-integer formulation suitable for instances where $\lambda \in [0,1]$. However, this formulation becomes invalid for cases where $\lambda > 1$. Consequently, we propose a general-case representation as a bilevel problem formulation. This includes considering the limit as $\lambda \rightarrow \infty$, which aligns with a valid formulation for the generalized-center problem (see equation~\eqref{eq:GC}). We then convert the bilevel problem into a single-level valid formulation. Furthermore, we integrate the criterion for efficiency, as specified in equation~\eqref{eq:delta_constraint}, into our proposed models. Finally, we describe a method to obtain the set $\mathcal{PO}^{\alpha}_2$ by solving two distinct linear formulations.

\subsection{Problem formulation for \texorpdfstring{$\lambda\in[0,1]$}{}}

We present a mixed-integer formulation of the $\lambda$-cent-dian problem, (CD) in what follows, for the case $\lambda\in[0,1]$ by introducing the design variables $y_i$ and $x_e$ that represent the binary design decisions of constructing node $i$ and edge $e$, respectively. For each $w\in W$, a set of flow variables $f^w_a,\,a\in A$ is used to model a path between $w^s$ and $w^t$, if possible. The variable $f_a^w$ takes the value $1$ if the arc $a$ belongs to the path from $w^s$ to $w^t$, and $0$ otherwise. We consider an extra artificial arc $r=(w^s, w^t)$ and its corresponding flow variable $f^w_r$ to model the alternative mode. The variable $f^w_r$ takes the value $1$ if the length of the shortest path from $w^s$ to $w^t$ in the designed subgraph is greater than $u^w$ and 0 otherwise. In other words, $f^w_r=1$ represents the binary decision of the demand taking the alternative mode. Finally, we consider the continuous variable $\gamma^{max}$, taking the value of the maximum distance of any O/D pair in the graph. 
\begin{align}
	\mbox{(CD)} \,\, \min\, & \,\, \lambda\gamma^{max} + (1-\lambda)\frac{1}{G}\sum_{w\in W} g^w\left( \sum_{a \in A^w} d_a\,f_a^w +u^w\,f_r^w \right) \label{cons:fobj}\\
	\text{ s.t.} & \,\, \sum\limits_{e\in E}c_e\,x_e +\sum\limits_{i\in N}b_i\,y_i \leq \alpha\,C_{total}, \label{cons:budget}\\
	& \,\, x_e\leq y_i, \hspace{6.7cm} e\in E, \, i\in e,\label{cons:location}\\
    \begin{split}
    & \,\, f_r^w+\sum_{a \in \delta_w^+(i)} f^w_a -\sum_{a \in \delta_w^-(i)} f^w_a = \\
	& \,\, = \begin{cases}
		1, &\text{if $i = w^s$},\\
		-1, &\text{if $i = w^t$}, \\
		0, & \text{otherwise}, 
	\end{cases} \hspace{0.2cm} w=(w^s,w^t)\in W,\, i\in N, \label{cons:flow}
    \end{split}\\
	& \,\, f_a^w + f_{\hat{a}}^w \leq x_e,\hspace{2.5cm} w\in W,\, a=(i,j) \in A^w: e=\{i,j\},\label{cons:capacity}\\
	& \,\, \sum_{a \in A} d_a\,f_a^w +u^w\,f_r^w\leq \gamma^{max}, \hspace{4.6cm} w\in W, \label{cons:center}\\
	& \,\, x_e, \, y_i, \, f_r^w, \,f_a^w \, \in \{0,1\},\,\, \gamma^{max}\geq0, \hspace{0.4cm} i\in N, \, e\in E, \, a \in A, \, w\in W. \label{cons:var}
\end{align}

The objective function \eqref{cons:fobj} to be minimized represents a convex combination between the center and the median objectives. Constraint \eqref{cons:budget} limits the total construction cost, being $\alpha\in (0,1]$. Constraint \eqref{cons:location} ensures that if an edge is constructed, then its terminal nodes are constructed as well. For each pair $w$, constraints \eqref{cons:flow},  guarantee flow conservation and demand satisfaction. Constraints \eqref{cons:capacity} are named \textit{capacity constraints} and force each edge to be used in at most one direction by each O/D pair whenever such edge is built. Constraints \eqref{cons:center}, referenced as \textit{maximum distance constraints}, determines the maximum length of the shortest path between all pairs $w=(w^s,w^t) \in W$. The structure of the objective function ensures that variable $f_r^w$ assumes the value of $0$ only if a path between $w^s$ and $w^t$ exists with a length not exceeding $u^w$ in the prospective network. This path is symbolized by the variables $f_r^w$. Subsequently, $\gamma^{max}$ denotes the greatest $\ell_{\mathcal{S}}(w)$ within the set $W$. Finally, constraints \eqref{cons:var} specify that all variables are binary, with the exception of variable $\gamma^{max}$.

For $\lambda \in [0,1)$, the objective function of the above formulation is composed of two non-negative coefficients ($\lambda$ and $(1-\lambda)$) that multiply increasing functions of travel times. In consequence, at any optimal solution, the demand always chooses the shortest path without the necessity of enforcing this condition. That is, optimal solution vectors $\boldsymbol{f}^w$ correspond to shortest paths in the solution network. 

For $\lambda \ge 1$, the variables $f$ are no longer mandated to represent the shortest path. In the specific instance where $\lambda = 1$, the solution sub-network of (CD) serves as a center of the network, preserving the center objective's correctness. This suggests that (CD) remains a valid formulation for the center problem. However, when $\lambda > 1$, this no longer holds true. As the term ($1-\lambda$) is negative, each O/D pair $w$ within this formulation will opt for an arbitrary path between $w^s$ to $w^t$ in the designed network. 

\begin{example}
Let us consider the same potential network and its associated data from Example $3$ of \cite{bucarey2024on} with $\alpha\,C_{total}=100$ and $\lambda=50\cdot10^4$. One of the optimal $\lambda$-cent-dian solution networks for formulation (CD) is

\begin{center}
	\begin{tikzpicture}[scale=0.45, auto,swap]
		\foreach \pos/\name in {{(1.5,2.5)/v_1}, {(3,4)/v_2}, {(3,1)/v_3}, {(5,4)/v_4}, {(5,1)/v_5}, {(6.5,2.5)/v_6}}
		\node[vertex] (\name) at \pos {$\name$};
		\foreach \source/\dest in {v_2/v_1, v_1/v_3, v_3/v_2, v_3/v_5, v_5/v_4, v_5/v_6}
		\path[edge] (\source) -- (\dest);
	\end{tikzpicture}
\end{center}
with the flow vector $f^w_r=1,\,w\in\{(1,2),(4,1)\}$ and $f^w_r=0,\,w\in\{(2,6)\}$. The path selected for $w=(2,6)$ is $[2,1,3,5,6]$ but the path $[2,3,5,6]$ is shorter than the selected one. Besides, this solution does not assign the pairs $(1,2)$ and $(4,1)$ to the existing paths in the network although they are shorter than the utilities $u^{(1,2)}=92$ and $u^{(4,1)}=92$. Then, the objective value is $92-91.90$. Nevertheless, this is not the solution network that should be obtained, as you can check in Example 3.
\end{example}

\subsection{Problem formulation for \texorpdfstring{$\lambda>1$}{}. General problem formulation}

We introduce a bilevel formulation to ensure that each $w$ selects the shortest path in the solution network for any value of $\lambda\ge 0$, especially for cases where $\lambda >1$. This is called the \textit{bilevel $\lambda$-cent-dian} formulation and is denoted as (BCD).

\begin{align}
	\mbox{(BCD)}\, \min & \quad \lambda\gamma^{max} + (1-\lambda)\frac{1}{G}\sum_{w\in W} g^w\left( \sum_{a \in A^w} d_a\,f_a^w +u^w\,f_r^w \right)\\
	\text{ s.t.} & \quad \sum\limits_{e\in E}c_e\,x_e +\sum\limits_{i\in N}b_i\,y_i \leq \alpha\,C_{total}, \\
	& \quad x_e\leq y_i, \hspace{6.6cm} e\in E, \, i\in e,\\
	& \quad \sum_{a \in A} d_a\,f_a^w +u^w\,f_r^w\leq \gamma^{max}, \hspace{4.5cm} w\in W, \\
	& \quad \boldsymbol{f}'^w\in sol(\mbox{(F)}^w), \hspace{6.2cm} w\in W,\\
	&\quad x_e, \, y_i, \, f_a^w \, \in \{0,1\},\quad \gamma^{max}\geq0, \hspace{0.6cm} i\in N, \, e\in E, \, a \in A, \, w\in W,
\end{align}

\noindent where
\begin{equation}\label{eq:CD_2_bilevel}
	\mbox{(F)}^w = \arg\min\limits_{\boldsymbol{f}'}\left\{\left( \sum_{a \in A^w} d_a\,f_a'^w +u^w\,f_r'^w \right):\, \text{ s.t.}
	\left.\begin{cases}
		\text{const.}(\ref{cons:flow})-(\ref{cons:capacity}),\kern-1cm\\
		f_r'^w,f_a'^w \, \in \{0,1\},\\
        a\in A^w\kern-0.3cm
	\end{cases}\right\}\right\}.
\end{equation}

The bilevel formulation (BCD) consists in an upper-level task that designs a network to minimize the $\lambda$-cent-dian function, and a lower-level task for each O/D pair $w\in W$ that aims to identify the shortest path between $w^s$ and $w^t$ within the designed network. The optimal solutions for the lower-level problems are articulated in equation \eqref{eq:CD_2_bilevel} and are denoted as $\mbox{(F)}^w$.

We reformulate the bilevel problem (BCD) as a single-level problem by imposing the optimality conditions of each problem (F)$^w$ as constraints of the upper-level problem. To do so, we consider the dual of each problem (F)$^w$, denoted by (DF)$^w$:
\begin{align}
	\mbox{(DF)}^w\quad \max\limits_{\boldsymbol{\nu},\boldsymbol{\sigma}} & \quad\nu^w_{w^s} - \sum_{e \in E} x_e\,\sigma^w_e \\
	\mbox{s.t.}& \quad \nu^w_i - \nu^w_j - \sigma^w_e \leq d_a,\hspace{2.3cm} a=(i,j) \in A^w : e=\{i,j\}, \label{cons:duality1}\\
	& \quad \nu^w_{w^s} \leq u^w, \label{cons:duality2}\\
	& \quad \sigma^w_e \geq 0, \hspace{7cm} e\in E^w.
\end{align}

\noindent As it is clear from the context, we omit the index $w$. Variables $\nu_i,i\in N$, are the dual variables related to the flow constraints \eqref{cons:flow} and $\sigma_e,\,e\in E$, are the dual variables corresponding to the capacity constraints \eqref{cons:capacity}. Given that the set of flow constraints contains one linear dependent constraint, we set $\nu_{w^t} = 0$.

Since $d_a$ and $u^w$ are positive parameters, each $(F)^w$ has optimal finite solutions. In consequence, by strong duality, a feasible vector $\boldsymbol{f}^w$ for (F)$^w$ is optimal if and only if there exist feasible vectors $\boldsymbol{\nu}$ and $\boldsymbol{\sigma}$ for $\mbox{(DF)}^w$ such that:
$$\sum\limits_{a \in A} d_a\,f_a +u\,f_r = \nu_{w^s} - \sum\limits_{e \in E} x_e\,\sigma_e.$$
\noindent Then, (BCD) can be cast as a single-level non-linear optimization problem:
\begin{align}
	\mbox{(BCD\_R)} \,\, \min & \quad \lambda\gamma^{max} + (1-\lambda)\frac{1}{G}\sum_{w\in W} g^w\left( \sum_{a \in A^w} d_a\,f_a^w +u^w\,f_r^w \right)\nonumber\\
	\text{ s.t.} & \quad \sum\limits_{e\in E}c_e\,x_e +\sum\limits_{i\in N}b_i\,y_i \leq \alpha\,C_{total}, \nonumber\\
	&\quad x_e\leq y_i, \quad\quad e\in E, \, i\in e,\nonumber\\
	&\quad \sum_{a \in A} d_a\,f_a^w +u^w\,f_r^w\leq \gamma^{max}, \hspace{4cm} w\in W,\nonumber \\
	\begin{split}
    & \quad f_r^w+\sum_{a \in \delta_w^+(i)} f^w_a -\sum_{a \in \delta_w^-(i)} f^w_a = \\
	& \quad =\begin{cases}
		1, &\text{if $i = w^s$},\\
		0, & \text{if } i\notin\{w^s,w^t\}, 
	\end{cases} \quad w=(w^s,w^t)\in W,\, i\in N,\nonumber
    \end{split}\\
	&\quad f^w_a + f^w_{\hat{a}} \leq x_e,\hspace{1.7cm} w\in W,\, e=\{i,j\} \in E^w: a=(i,j),\nonumber\\
	&\quad \nu^w_i - \nu^w_j - \sigma^w_e \leq d_a,\hspace{0.9cm} w\in W,\, a=(i,j) \in A^w : e=\{i,j\},\nonumber\\
	&\quad \nu^w_{w^s} \leq u^w, \hspace{6.7cm} w\in W,\nonumber\\
	&\quad \sum\limits_{a \in A^w} d_a\,f^w_a +u^w\,f^w_r = \nu^w_{w^s} - \sum\limits_{e \in E^w} x_e\,\sigma^w_e, \hspace{2cm} w\in W, \label{cons:strongdual}\\
	\begin{split}
    & \quad x_e, \, y_i, \, f_r^w, \, f_a^w \, \in \{0,1\},\, \gamma^{max},\,\sigma^w_e\geq0, \, \\
     & \quad i\in N, \, e\in E, \, a \in A, \, w\in W.\label{eq_BCD:variable}
     \end{split}
\end{align}

The model above is a mixed integer non-linear problem that can be linearized by replacing the non-linear terms in \eqref{cons:strongdual}. We introduce the set of variables $\xi^w_e$, representing the product $x_e\,\sigma^w_e$ and we linearize them by using the following McCormick inequalities (see \cite{mccormick1976computability}):
\begin{align}
	\xi_e \le x_e \Sigma_e,\\
	\xi_e \le \sigma_e,\\
	\xi_e \ge \sigma_e - \Sigma_e(1-x_e),\\
	\xi_e \ge 0,
\end{align}
\noindent where $\Sigma_e$ is an upper bound of $\sigma_e$. 

The computational efficiency of the McCormick reformulation relies on the availability of tight bounds of $\sigma^w_e$. To get such bounds, we consider the assumption that for each $w\in W$, $d_{\mathcal{N}}(w)\leq u^w$, meaning that there exists a possible path to be designed for each pair that improves its private utility. This is not a restrictive assumption since a pair not satisfying this condition will always prefer the private mode and can be eliminated from the analysis. Under this mild assumption, Propositions \ref{prop:validbound} and \ref{prop:best_bound} establish an efficient way to compute the best bound on the dual variables $\sigma^w_e$.

\begin{proposition}\label{prop:validbound}
	The quantity $\Sigma^w_e =u^w-d_{\mathcal{N}}(w)$ is a valid bound for $\sigma_e^w$ in (BCD\_R).
\end{proposition}

\begin{proof}
	For each $w\in W$, rewriting \eqref{cons:strongdual}, we have that
	$$\sum\limits_{e \in E^w} x_e\,\sigma^w_e = \nu^w_{w^s} - \left(\sum\limits_{a \in A^w} d_a\,f^w_a + u^w\,f^w_r\right). $$
	Using \eqref{cons:duality2} and that $\sum\limits_{a \in A^w} d_a\,f^w_a + u^w\,f^w_r\geq d_{\mathcal{N}}(w)$, then
	$$\sum\limits_{e \in E^w} x_e\,\sigma^w_e \leq u^w -d_{\mathcal{N}}(w).$$
	Thus, we can conclude that $\sigma^w_e \leq u^w -d_{\mathcal{N}}(w)$.
\end{proof}

\begin{proposition}\label{prop:best_bound}
	The bound given in Proposition \ref{prop:validbound} is tight.
\end{proposition}

\begin{proof}
	Let us fix a pair $w\in W$ and consider the particular situation for which there exists $e=\{w^s,w^t\}\in E^w\subseteq E$ and $d_{\mathcal{N}}(w)=d_{(w^s,w^t)}$. Let us suppose that there exists a better bound for $\sigma_{\{w^s,w^t\}}^w$ than $u^w-d_{\mathcal{N}}(w)$. That is, $\sigma_{\{w^s,w^t\}}^w\leq u^w-d_{\mathcal{N}}(w)-\epsilon$. By \eqref{cons:duality1}, 
	$\nu_{w_s}^w\leq d_a + \sigma_{\{w^s,w^t\}}^w \leq u^w-\epsilon$. Then, according to \eqref{cons:strongdual}, 
	\begin{align}
		\nu_{w_s}^w = \sum_{a\in A^w}d_a\,f^w_a+u^w\,f^w_r +\sum_{e\in E^w}x_e\,\sigma_e^w\leq u^w-\epsilon.
	\end{align}\label{equation:prop6}
	In order to satisfy the previous inequality, $f^w_r$ has to be set to $0$, which forces a solution network to have a path for $w$ (the corresponding flow variables take value equal to $1$). This situation can result in an infeasible solution by \eqref{eq:CD_2_bilevel} if the length of the assigned path is larger than its $u^w$ or even if there is no path.
	
	Let us show it with a small example. Let $\mathcal{N}$ be the following potential network.
	
	\begin{table}[ht]
	\begin{minipage}{0.45\textwidth}
			\begin{tabular}{cccc}
				\hline
				Origin & Destination & $u^w$ & $g^w$ \\
				\hline
				1 & 2 & 24 & 181 \\
				1 & 4 & 34 & 168 \\
				2 & 4 & 20 & 43 \\
				3 & 2 & 32 & 121\\ 
				\hline
			\end{tabular}
			\label{table:propositionbestbound}
	\end{minipage}
	\begin{minipage}{0.45\textwidth}
		\centering
		\begin{tikzpicture}[scale=1, auto,swap]
			\foreach \pos/\name in {{(3,4)/v_1}, {(3,1)/v_3}, {(6,4)/v_2}, {(6,1)/v_4}}
			\node[vertex] (\name) at \pos {$\name$};
			\node[vertex2]  at (3,0.5) {$10$};
			\node[vertex2]  at (3,4.5) {$8$};
			\node[vertex2]  at (6,0.5) {$8$};
			\node[vertex2]  at (6,4.5) {$7$};
			\foreach \source/\dest/\weight in {v_2/v_1/{(12,12)}, v_1/v_3/{(14,14)}, v_1/v_4/{(17,17)}, v_4/v_2/{(10,10)}, v_3/v_4/{(6,6)}}
			\path[edge] (\source) -- node[weight] {$\weight$} (\dest);
		\end{tikzpicture}
	\end{minipage}
			\end{table}
	
	Being $\lambda = 20$ and $\alpha\,C_{total}=63$, the optimal solution for \textit{(Bilevel-CD)} is
	
	\begin{center}
		\begin{tikzpicture}[scale=0.4, auto,swap]
			\foreach \pos/\name in {{(3,4)/v_1}, {(3,1)/v_3}, {(6,4)/v_2}, {(6,1)/v_4}}
			\node[vertex] (\name) at \pos {$\name$};
			\foreach \source/\dest in {v_1/v_3, v_2/v_4, v_3/v_4}
			\path[edge] (\source) -- (\dest);
		\end{tikzpicture}    
	\end{center}
	which means that $x_{\{1,2\}}=0$. 
	By Proposition \ref{prop:best_bound}, it is set $f_{r}^{(1,2)}=0$, which forces to assign $w$ to the solution network. This solution network contains a path for $w$ but it is larger than its $u^w$, which results in an infeasibility by \eqref{eq:CD_2_bilevel} for not taking the optimal mode.
\end{proof}

To conclude, as we discussed, the generalized-center tends to be suboptimal in efficiency and is not necessarily characterized as a center network. We introduced a constraint to enhance the efficiency of the generalized-center solution. Specifically, for cases where $\lambda>1$, we augment the (BCD) formulation with the following constraint:
\begin{equation}
	\frac{1}{G}\sum_{w\in W} g^w\left( \sum_{a \in A^w} d_a\,f_a^w +u^w\,f_r^w \right) \leq(1+\Delta)F_m(\mathcal{S}_m),
\end{equation}
\noindent being $F_m(\mathcal{S}_m)$ the objective value of the median network (the objective value for (CD) for $\lambda=0$).

\subsection{An algorithm to describe the set \texorpdfstring{$\mathcal{PO}^{\alpha}_2$}{}}\label{subsec:formulation_PO2}

The set $\mathcal{PO}^{\alpha}_2$ can be exhaustively described as a function of $\lambda\in(0,1)$ with minimization of the weighted function $\Bar{H}_{\lambda}(\,\cdot\,)$. In the case of a non-unique optimal solution, we select the one that has minimum $H_{\lambda}(\,\cdot\,)$ value. That is, firstly, to minimize $\Bar{H}_{\lambda}(\,\cdot\,)$ we solve the problem
\begin{align}
	\quad \min\, & \quad \mu \label{eq:lexmin1_obj} \\
	\text{ s.t.} & \quad \mu\geq \lambda\,F_c(\mathcal{S}), \label{cons:lexmin1_2}\\
	& \quad \mu\geq (1-\lambda)\,F_m(\mathcal{S}), \label{cons:lexmin1_3}\\
	& \quad \mathcal{S}\in\mathcal{N}^{\alpha}, \label{cons:lexmin1_4}\\
	& \quad \mu\geq 0. \label{cons:lexmin1_5}
\end{align}

Then, being $V^*$ the objective value of the problem formulation \eqref{eq:lexmin1_obj}-\eqref{cons:lexmin1_5}, the following second problem is solved for regulation purposes:
\begin{align}
	\quad \min\, & \quad \lambda F_c(\mathcal{S}) + (1-\lambda)F_m(\mathcal{S}) \label{eq:lexmin3_obj} \\
	\text{ s.t.} & \quad \lambda\, F_c(\mathcal{S}) \leq V^*, \label{eq_C3:lexmin3_1}\\
	& \quad (1-\lambda)\, F_m(\mathcal{S}) \leq V^*, \label{cons:lexmin3_2}\\
	& \quad \mathcal{S}\in\mathcal{N}^{\alpha},\label{cons:lexmin3_3}
\end{align}

In terms of sets $N$, $E$ and $W$, the previous formulations are equivalent to the following ones
\begin{align}
	\mbox{(MCD\_1)} \quad \min\, & \quad \mu \\
	\text{ s.t.} & \quad \mu\geq \lambda\,\left( \sum_{a \in A^w} d_a\,f_a^w +u^w\,f_r^w \right), \hspace{3cm} w\in W,\\
	& \quad \mu\geq (1-\lambda)\,\frac{1}{G}\sum_{w\in W} g^w\left( \sum_{a \in A^w} d_a\,f_a^w +u^w\,f_r^w \right), \\
	& \quad \mu\geq 0, \\
	& \quad \text{constraints }\eqref{cons:budget},\eqref{cons:location},\eqref{cons:flow},\eqref{cons:capacity},\eqref{cons:var}
\end{align}
\begin{align}
	\mbox{(MCD\_2)} \quad \min\, & \quad \lambda \left( \sum\limits_{a \in A^w} d_a\,f_a^w +u^w\,f_r^w \right) + (1-\lambda)\frac{1}{G}\sum\limits_{w\in W} g^w\left( \sum\limits_{a \in A^w} d_a\,f_a^w +u^w f_r^w \right) \\
	\text{ s.t.} & \quad 0\leq V^* - \lambda\,\left( \sum_{a \in A^w} d_a\,f_a^w +u^w\,f_r^w \right), \hspace{2cm} w\in W,\\
	& \quad 0\leq V^* - (1-\lambda)\,\frac{1}{G}\sum_{w\in W} g^w\left( \sum_{a \in A^w} d_a\,f_a^w +u^w\,f_r^w \right), \\
	& \quad \mu\geq 0, \\
	& \quad \text{constraints }\eqref{cons:budget},\eqref{cons:location},\eqref{cons:flow},\eqref{cons:capacity},\eqref{cons:var}.
\end{align}

To obtain all the existing correspondences between all the solution networks in $\mathcal{PO}^{\alpha}_2$ and its associated value $\lambda\in(0,1)$, we can use the bisection method for $\lambda$ parameter in the interval $(0,1)$. 

\section{Quality of the solutions with a numerical illustration}\label{sec:delta_results}

In this section, we illustrate the quality of the solutions for the different solution concepts developed in Section \ref{sec:problem_definition} through an example. The instance used in this part has been generated randomly in the same manner as in Section \ref{subsec:comp_results}, being the network resultant of a planar graph consisting of 20 nodes, shown in Figure \ref{fig:potential_graph}.

\begin{figure}[ht]
	\centering
	\hspace*{-0.3cm}\includegraphics[scale=0.35]{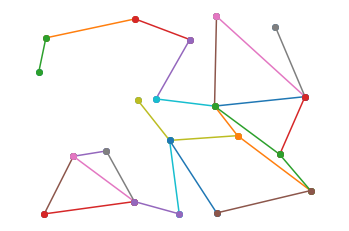}
	\caption{Potential graph used in Section \ref{sec:delta_results}.}
	\label{fig:potential_graph}
\end{figure}

The quality of the solution $\mathcal{S}$ is presented in terms of:
\begin{itemize}
	\item the minimum length of a path among all the pairs: $\ell_{min}({\mathcal{S}}) = \min_w \ell_{\mathcal{S}}(w)$, 
	\item the maximum length of a path among all the pairs: $\ell_{max}({\mathcal{S}}) = \min_w \ell_{\mathcal{S}}(w)$, 
	\item the average length of the paths: $\overline{\ell}(\mathcal{S}) = \frac{1}{|W|}\sum_w \ell_{\mathcal{S}}(w)$, 
	\item the mean absolute difference of the path's length:\\ MAD$ (\mathcal{S}) = \frac{1}{2}\sum\limits_{w\in W}\sum\limits_{\substack{w'\in W:\\ w\neq w'}} g^wg^{w'}|\ell_{\mathcal{S}}(w) -\ell_{\mathcal{S}}(w')|$,\\
    which has just been studied and taken into account in \cite{lopez2003sum} and \cite{mesa2003improved},
	\item the O/D pairs served refers to the number of O/D pairs that prefer the solution network over the one provided by the private mode. This is denoted by O/D\%.
\end{itemize}

\begin{table}[ht]
	\centering
	\begin{tabular}{cc c r r r r r}
		\hline
		$\lambda$ & $\Delta$ (\%) &$\ell_{min}$ &$\ell_{max}$ & $\overline{\ell}$ &  MAD& O/D\%\\ \hline
		0    & -    & 4 & 100 & 45.037  & 12.351 & 28.187 \\
		0.25 & -    & 4 & 100 & 45.037  & 12.351 & 28.187 \\
		0.5  & -    & 4 & 96  & 47.149 & 12.375 & 27.516 \\
		0.75 & -    & 4 & 96  & 47.149 & 12.375 & 27.516 \\\hline
		1    & -    & 4 & 96  & 47.383 & 12.187 & 6.375  \\
		1    & 0    & 4 & 100 & 45.037  & 12.351 & 28.187 \\
		1    & 3 & 4 & 100 & 45.981 & 12.322 & 24.832 \\
		1    & 5 & 4 & 96  & 47.149 & 12.375 & 23.154 \\
		1    & 10  & 4 & 96  & 47.383 & 12.187 & 8.389  \\
		1    & 15 & 4 & 96  & 47.383 & 12.187 & 8.389  \\\hline
		500  & -    & 4 & 96  & 48.655 & 12.535 & 21.476 \\
		500  & 0    & 4 & 100 & 45.037  & 12.351 & 28.187 \\
		500  & 3 & 5 & 100 & 46.386  & 12.279 & 24.832 \\
		500  & 5 & 4 & 96  & 47.283 & 12.859 & 26.174 \\
		500  & 10  & 4 & 96  & 48.655 & 12.535 & 21.476 \\
		500  & 15 & 4 & 96  & 48.655 & 12.535 & 21.476
		\\\hline
		lex &-& 4 &96 &47.149 & 12.375&  27.516 \\\hline
	\end{tabular}
	\caption{Quality of solutions for  $\lambda$-cent-dian and generalized-center problems using a random instance with $|N|=20$ and $\alpha=0.4$.}
	\label{t:numerical_illustration}
\end{table}

Table \ref{t:numerical_illustration} summarizes the different values of the performance index mentioned above and is arranged as follows. The first set of rows corresponds to the $\lambda$-cent-dian problem for $\lambda \in \{0,0.25,0.5, 0.75\}$. Secondly, we consider the center problem ($\lambda =1$). In the third set of rows, the approximation of the generalized-center concept with $\lambda=500$ is considered. For the center and generalized-center problems, the effect of adding the efficiency constraint \eqref{eq:delta_constraint} with $\Delta\in\{0\%,3\%,5\%, 10\%,15\%\}$ is considered. In the last row, we consider the lexicographic cent-dian \eqref{eq:lexmin}. Note that for the particular case with $\lambda=1$ and $\Delta=0$, we are solving the inverse lexicographic problem: first, we solve the median problem by optimizing $F_m(\cdot)$, and subsequently, we seek the best center among the optimal solutions of the median. We have solved formulations (CD) and (BCD) exposed in Section \ref{sec:problem_formulation} with the direct use of Python 3.8 and the optimization solver \texttt{CPLEX} 12.10. 

First, note that only two different solutions are obtained by solving the problem for range $\lambda\in(0,1)$. These solutions belong to the set $\mathcal{PO}_2^{\alpha}$. In the first four rows of Table \ref{t:numerical_illustration} is noted the compromise between the median and the center objectives. For $\lambda \in \{0, 0.25\}$, the maximum travel time ($\ell_{max}$) is 100, which decreases to 96 when $\lambda \in \{0.5, 0.75\}$ is considered. As it is expected, the opposite effect is observed with the median objective $\overline{\ell}(\mathcal{S})$, which increases from 45.037 to 47.149. We observe the same effect with the percentage of O/D pairs served. 

For $\lambda=1$ and $\Delta=0$,  the formulation seeks the median solution with the minimum center value.   Similarly, when $\lambda=500$ and $\Delta=0$, the optimal solution corresponds to the median network with the lowest generalized-center value. These solution networks are exactly the solution for the configuration with $\lambda\in \{0, 0.25\}$. Comparing configurations with $\Delta>0$ for both $\lambda=1$ and $\lambda=500$, we observe that with, $\lambda=1$ the average travel time increases, being greater than in the median solution (as expected) and smaller than in the generalized-center solutions. Additionally, in both cases, employing the efficiency constraint results in more O/D pairs being served than when it is not used, provided that the parameter $\Delta$ is not excessively large.  Furthermore, for a fixed value of $\Delta$, the percentage of O/D pairs served by the generalized-center solution is either greater than or equal to that served by the center solution. Regarding the mean absolute difference measure, no distinct trend can be identified; we simply observe that, in this case, the center solution has the lowest value for this measure.

Besides, we observe that comparing the center solutions between them, there are some dominated ones. More specifically, the center solution without the efficiency constraint ($\Delta$=\,-\,) is completely dominated by the solution where this constraint is used with $\Delta=15$. On the other hand, a similar situation arises when comparing the generalized-center solutions among themselves.

Making a general comparison, we highlight that there are some generalized-center solutions non-dominated by the rest. The generalized-center solution with $\Delta=3$ has smaller values of M.A.D than any of $\lambda$-cent-dian solutions with $\lambda<1$. If we compare it with the center solutions, we see that two situations can occur: i) the center solution has a smaller M.A.D value but a bigger average value and a lower percentage of demand covered, or ii) if the center has a worse M.A.D value, then it also has better values for some of the inequality measures evaluated.

In a general comparison, we emphasize that some generalized-center solutions are not dominated by the rest of the solutions in Table \ref{t:numerical_illustration}. The generalized-center solution with $\Delta=3$ has smaller values of M.A.D than any of the $\lambda$-cent-dian solutions with $\lambda<1$. When comparing it with the center solutions, two situations may arise: i) the center solution has a smaller M.A.D value but a larger Average value and a lower percentage of demand covered, or ii) if the center solution has a worse M.A.D value, it also has better values for some of the inequality measures evaluated.

\section{Algorithmic discussion for \texorpdfstring{$\lambda\in[0,1]$}{}}\label{sec:alg_discussion}

In this section we present how for $\lambda \in [0,1]$, the formulation (CD) possesses the suitable separable structure needed for a Benders decomposition approach. Indeed, when the design variables of (CD), namely $x$, $y$ and $\gamma^{max}$, are fixed, the problem can be divided into $|W|$ sub-problems. Each of these sub-problems establishes the flow variables $f^w$ as a linear problem by relaxing their integrality condition. 

However, it is important to note that this property is not present in (BCD), where the flow variables must also meet a non-convex optimality condition. Since (CD) is not a valid formulation for $\lambda >1$, the Benders decomposition developed in this section is only appropriate for $\lambda \in [0,1]$.

\subsection{Benders Decomposition with facet-defining cuts for \texorpdfstring{$\lambda\in[0,1]$}{}}\label{subsec:benders_approach}

Formulation (CD) involves a large number of flow variables when the set of O/D pairs is extensive. To address this issue, we explore a stabilized Benders decomposition based on the concepts presented in \cite{conforti2019facet2} to generate \textit{facet-defining cuts}. This method has been previously examined and developed for covering problems in \cite{cordeau2019benders} and \cite{bucarey2022benders}. 
We investigate an implementation of the branch-and-Benders-cut algorithm (\texttt{B\&BC}) for the case $\lambda\in[0,1]$.

To generate such cuts, we first need to relax the integrality condition on the flow variables $f^w_a$ and $f^w_r$. Proposition \ref{prop_C3:projection} shows that this can be done without loss of generality. Let (CD\_R) denote the formulation (CD) in which constraints $f_a^w\in\{0,1\},w\in W, a \in A\cup\{r\}$ are replaced by non-negativity constraints, i.e.
\begin{equation}
	f^w_a \geq 0, w \in W, a \in A\cup\{r\}.
\end{equation}

Let $Q$ be a set of points $(\boldsymbol{x},\boldsymbol{z}) \in \mathbb{R}^n \times \mathbb{R}^m$. Then, the projection of $Q$ onto the $x$-space, denoted by $Proj_x(Q)$, is the set of points given by: $Proj_{\boldsymbol{x}} (Q) = \{\boldsymbol{x} \in \mathbb{R}^m: (\boldsymbol{x},\boldsymbol{z}) \in Q \text{ for some } \boldsymbol{z} \in \mathbb{R}^n\}.$ Let us denote by $\mathcal{F}(CD)$ the set of feasible points of formulation $(CD)$.

\begin{proposition}\label{prop_C3:projection}
	The projections of (CD) and (CD\_R) onto the $\boldsymbol{f}$-space coincide.
	$$Proj_{\boldsymbol{x},\boldsymbol{y},\boldsymbol{\gamma^{max}}}(\mathcal{F}(CD)) = Proj_{\boldsymbol{x},\boldsymbol{y},\boldsymbol{\gamma^{max}}}(\mathcal{F}(CD\_R)).$$
\end{proposition}
\begin{proof} 
	\noindent First, $\mathcal{F}(CD) \subseteq \mathcal{F}(CD\_R)$ implies $Proj_{\boldsymbol{x},\boldsymbol{y},\boldsymbol{\gamma^{max}}}(\mathcal{F}(CD)) \subseteq Proj_{\boldsymbol{x},\boldsymbol{y},\boldsymbol{\gamma^{max}}}(\mathcal{F}(CD\_R))$. Second, let $(\boldsymbol{x},\boldsymbol{y},\boldsymbol{\gamma^{max}})$ be a point belonging to $Proj_{\boldsymbol{x},\boldsymbol{y},\boldsymbol{\gamma^{max}}}(\mathcal{F}(CD\_R))$. Due to the structure of the objective function, every O/D pair will select one of the two modes, the shortest one. That is, $f^w_r$ will take the value of $0$ or $1$. If $f_r^w = 1$ then $\boldsymbol{f}^w=0$. In the case where $f_r^w=0$, there exists a flow $f^w_a \geq 0$ satisfying \eqref{cons:flow} and \eqref{cons:capacity} that can be decomposed into a convex combination of flows on paths from $w^s$ to $w^t$ plus possibly some cycles. Further,  $x_e =1$ for all edges belonging to these paths which have the same length. Given that the flow $f_a^w$ also satisfies \eqref{cons:center}, then a flow of value $1$ on one of the paths in the convex combination must satisfy this constraint. Hence, by taking $f_a^w$ equal to $1$ for the arcs belonging to this path and to $0$ otherwise, we show that $(\boldsymbol{x},\boldsymbol{y},\boldsymbol{\gamma^{max}})$ also belongs to $Proj_{\boldsymbol{x},\boldsymbol{y},\boldsymbol{\gamma^{max}}}(\mathcal{F}(CD))$. 
\end{proof}

Based on Proposition \ref{prop_C3:projection}, we propose a Benders decomposition where variables $f^w_a$ and $f^w_r$ are projected out from the model and replaced by Benders facet-defining cuts, generated on-the-fly. Following the Benders decomposition Theory, since flow variables appear in the objective function of (CD), we have introduced an incumbent variable $\zeta^w$ for each O/D pair $w\in W$ to denote the expression $\sum_{a\in A^w}d_a\,f_a^w+u^w\,f_r^w$. Then, the master problem that we solve is
\begin{align}
	\mbox{(M\_CD)} \quad \min & \quad \lambda\gamma^{max} + (1-\lambda)\frac{1}{G}\sum_{w\in W} g^w\zeta^w\\
	\text{ s.t.} & \quad \sum\limits_{e\in E}c_e\,x_e +\sum\limits_{i\in N}b_i\,y_i \leq \alpha\,C_{total},\\
	& \quad x_e\leq y_i, \hspace{5cm} e\in E, \, i\in\{e^s,e^t\},\\
	& \quad \zeta^w\leq\gamma^{max}, \hspace{6.4cm} w\in W,\\
	&\quad +\{\mbox{Benders Cuts }(\boldsymbol{x},\boldsymbol{y},\boldsymbol{f},\boldsymbol{\zeta})\}. \nonumber\\
	& \quad x_e, \, y_i\, \in \{0,1\}, \hspace{0.5cm} \gamma^{max},\zeta^w\geq 0, \hspace{0.9cm} i\in N, \, e\in E, \, w\in W.
\end{align}

In \cite{conforti2019facet2}, the authors derive a cut-generating LP, the optimal solution of which almost surely induces a facet-defining Benders feasibility cut. To achieve this, it is necessary to generate a feasible solution in the interior of the convex hull of the feasible domain. Then, this cut-generating LP scheme finds the {\it best} cut that lies in the convex combination of such interior point and the point to be separated. 

We obtain this interior point by first performing a preprocessing method to delete all the O/D pairs in the network that will not be served by any feasible solution of (CD). This is equivalent to eliminating all the O/D pairs that induce the implicit equality $f^w_r =1$. We then show that the convex hull of the resulting problem is full-dimensional by exposing a sufficient number of linearly independent feasible points. The average of such points is a point living in the interior of the convex hull of the feasible space.

We briefly describe the preprocessing scheme as follows. The first part involves constructing, for each pair $w\in W$, a subgraph $\mathcal{N}^w$ that includes only those nodes and edges present in any path from $w^s$ to $w^t$ that is shorter than or equal to $u^w$. In the second part, we identify the set of O/D pairs deemed too expensive to serve, denoted as $\Bar{W}$. Specifically, these are pairs without a path from $w^s$ to $w^t$ in $\mathcal{N}^w$ that satisfies both: i) its building cost is less than $\alpha\, C_{total}$; and ii) its length is less than $u^w$. The mode choice decision for this set is fixed to $f_r^w=1$ and deleted from $W$. For a more detailed explanation, we refer to \cite{bucarey2022benders}.

The following Proposition \ref{prop_C3:full_dimensional} and its proof gives us a way to compute $|N| + |E| + |W| + 2$ linearly independent points. In consequence, the average of these points is an interior point of the convex hull of $Proj_{\boldsymbol{x},\boldsymbol{y},\boldsymbol{\gamma^{max}},\boldsymbol{\zeta}}(\mathcal{F}(CD\_R))$.

\begin{proposition} \label{prop_C3:full_dimensional}
	After preprocessing, the convex hull of $Proj_{\boldsymbol{x},\boldsymbol{y},\boldsymbol{\gamma^{max}},\boldsymbol{\zeta}}(\mathcal{F}(\text{CD}\_R))$ is full-dimensional. 
\end{proposition}
\begin{proof}
	See Appendix \ref{appendix:proof}.
\end{proof}

Given an interior point \((\boldsymbol{x}^{in}, \boldsymbol{y}^{in}, \boldsymbol{\gamma^{max}}^{in}, \boldsymbol{\zeta}^{in})\) and an exterior point \((\boldsymbol{x}^{out}, \boldsymbol{y}^{out}, \boldsymbol{\gamma^{max}}^{out}, \boldsymbol{\zeta}^{out})\), which is a solution to the LP relaxation of the current restricted master problem \(Proj_{\boldsymbol{x},\boldsymbol{y},\boldsymbol{\gamma^{max}}}(\mathcal{F}(CD\_R))\), we generate a cut that induces either a facet or an improper face of the polyhedron defined by the LP relaxation of \(Proj_{\boldsymbol{x},\boldsymbol{y},\boldsymbol{\gamma^{max}},\boldsymbol{\zeta}}\mathcal{F}(CD)\). We denote the difference \(\boldsymbol{x}^{out} - \boldsymbol{x}^{in}\) by \(\Delta \boldsymbol{x}\) and define \(\Delta \boldsymbol{y}\), \(\Delta\boldsymbol{\gamma^{max}}\), and \(\Delta \boldsymbol{\zeta}\) analogously. The goal is to find the point furthest from the interior point that is feasible to the LP-relaxation of \(Proj_{\boldsymbol{x},\boldsymbol{y},\boldsymbol{z}}\mathcal{F}(CD)\) and lies on the line segment between the \textit{interior point} and the \textit{exterior point}. This point takes the form \((\boldsymbol{x}^{sep}, \boldsymbol{y}^{sep}, \boldsymbol{\gamma^{max}}^{sep}, \boldsymbol{\zeta}^{sep}) = (\boldsymbol{x}^{out}, \boldsymbol{y}^{out}, \boldsymbol{\gamma^{max}}^{out}, \boldsymbol{\zeta}^{out}) - \mu (\Delta \boldsymbol{x}, \Delta \boldsymbol{y}, \Delta \boldsymbol{\gamma^{max}}, \Delta \boldsymbol{\zeta})\). For each O/D pair \(w\), we cast the cut-generation problem as:
\begin{align}
	\mbox{(SP)}^w \, \min\limits_{\boldsymbol{f},\mu}&\, \mu\\
	\text{ s.t.} &\,\, f_r+\sum\limits_{a\in \delta^+(i)}f_a - \sum\limits_{a\in \delta^-(i)}f_a=
	\begin{cases}
		1, &\text{if $i = w^s$},\\
		0, & \text{otherwise}, 
	\end{cases} \hspace{1.8cm} i\in N,\label{eq:subpb_facets_flow}\\
	&\,\, f_a + f_{a'} \leq x_e^{out}-\mu\,\Delta x_e, \hspace{0.5cm} e=\{i,j\} \in E : a=(i,j), a'=(j,i),\label{eq:subpb_facets_loc_alloc}\\
	&\,\, \sum_{a \in A^w} d_a\, f_a + u\,f_r \leq\zeta^{out}-\mu\,\Delta\zeta, \label{eq:subpb_facets_utility}\\
	&\,\, 0 \leq \mu \leq 1, \\ 
	&\,\, f_a \geq 0, \hspace{7.7cm} a\in A.
\end{align}

In order to obtain the Benders feasibility cut we solve its associated dual. Given that (SP)$^w$ is always feasible ($\mu = 1$ is feasible) and that its optimal value is lower bounded by 0, then, both (SP)$^w$ and its associated dual problem have always finite optimal solutions. Whenever the optimal value of $\mu$ is 0, $(\boldsymbol{x}^{out}, \boldsymbol{y}^{out},\boldsymbol{\gamma^{max}}^{out}, \boldsymbol{\zeta}^{out})$ is feasible. A cut is added if the optimal value of the dual subproblem is strictly greater than 0. This cut has the form 
\begin{equation}
	- \sum\limits_{e\in E}\sigma_e\,x_e - \upsilon\,\zeta \leq -\phi_{w^s},\label{eq:cut-facets}
\end{equation}
\noindent being $\phi_i$, $\sigma_e$ and $\zeta$ the optimal dual variables associated to constraints \eqref{eq:subpb_facets_flow}, \eqref{eq:subpb_facets_loc_alloc} and \eqref{eq:subpb_facets_utility}, respectively.

\subsection{Computational results}\label{subsec:comp_results}

We conduct our experiments on a computer equipped with an Intel Core i$5$-$7300$ CPU processor, with $2.50$ gigahertz $4$-core, and $16$ gigabytes of RAM memory. The operating system used was 64-bit Windows 10. The codes were implemented in Python 3.8 and executed using the \texttt{CPLEX} 12.10 solver through its Python interface. The \texttt{CPLEX} parameters were set to their default values, and the model was optimized in single-threaded mode.

We generate random instances to test the computational experience as follows. We consider planar networks with a set of \(n=40\) nodes. Nodes are placed in a grid of \(n\) square cells, each measuring 10 units per side. For each cell, a point is randomly generated near the center. We construct a planar graph with the maximum number of edges, deleting each edge with a probability of 0.2. This procedure is replicated 10 times to ensure the number of nodes remains constant while the number of edges may vary, resulting in 10  underlying networks. Once a random instance \(\mathcal{N}\) is generated, construction costs \(b_i\), for \(i \in N\), are randomly generated according to a uniform distribution \(\mathcal{U}(7, 13)\), yielding an average cost of 10 monetary units per node. The construction cost of each edge \(e \in E\), denoted \(c_e\), is set to its Euclidean length, implying that building the links costs 1 monetary unit per length unit. The node and edge costs are rounded to integer numbers. Regarding the budget, we set \(\alpha \in \{0.25, 0.4\}\), meaning that the available budget equals 25\% or 50\% of the cost of building the entire underlying network considered. Additionally, to construct the set of O/D pairs \(W\), all possible pairs are taken into account, resulting in a set \(W\) composed of \(n(n - 1)\) elements. For each \(w \in W\), parameter \(u^w\) is set to twice the Euclidean length between \(w^s\) and \(w^t\). Finally, the demand \(g^w\) for each O/D pair \(w\) is randomly generated according to the uniform distribution \(\mathcal{U}(10, 300)\).

\begin{table}[ht]
	\begin{tabular}{ccccccc}
		\hline
		\multirow{2}{*}{$\alpha$}     & \multirow{2}{*}{Block} & \multicolumn{5}{c}{$\lambda$}      \\ \cline{3-7} 
		&                        & 0 & 0.25 & 0.5 & 0.75 & 1 \\ \hline
		\multirow{3}{*}{0.25} & \texttt{A}                      & 0  &  1    &  2   &   3   & 0  \\
		& \texttt{B}                      &  3 &   2   &   2  &    0  & 0  \\
		& \texttt{C}                      & 7  &    7  &    6 &    7  & 10  \\ \hline
		\multirow{3}{*}{0.4}  & \texttt{A}                      & 2  &  2    &   3  &   3   &  0 \\
		& \texttt{B}                      &  4 &   4   &    4 &    3  &  0 \\
		& \texttt{C}                      &   4&    4  &    3 &    4  & 10  \\ \hline
	\end{tabular}
	\caption{Instances solved for (CD) within a time limit of 1 hour.}
	\label{table:number_instances_solved}
\end{table}

Our preliminary experiments show that including cuts only at integer nodes of the \textit{branch-and-bound} tree is more efficient than including them in nodes with fractional solutions. Thus, in our experiments, we only separate integer solutions unless we specify the opposite. We used the \texttt{LazyCons\-traintCallback} function of \texttt{CPLEX} to separate integer solutions. Fractional solutions were separated using the \texttt{UserCutCallback} function.

We evaluate the implementation of the branch-and-Benders-cut algorithm (\texttt{B\&BC}) proposed in Section \ref{subsec:benders_approach} that generates facet-defining cuts (equation \eqref{eq:cut-facets}). We will use the nomenclature of \texttt{BD\_CW} to refer to this routine. In addition, we have added cuts to the root node because we have verified that this is profitable.  We compare our \texttt{BD\_CW} implementation with the direct use of \texttt{CPLEX} and with the automatic Benders procedure proposed by \texttt{CPLEX}, noted by \texttt{Auto\_BD}. \texttt{CPLEX} provides three configurations related to the decomposition. We have set the one that attempts to decompose the model strictly according to the decomposition provided by the user.

We perform the experiments with a limit of one hour of CPU time considering $10$ instances. Tables in this section show average values obtained for solution times in seconds, relative gaps in percent, and number of cuts needed. To determine these averages, for each value of $\lambda$ we have classified the instances into three blocks: block \texttt{A} contains those instances that have been solved to optimality with the three routines, block \texttt{B} contains those that have not been solved to optimality without any of the routines and block \texttt{C} is composed of those routines solved to optimality with one or two of routines proposed. Table \ref{table:number_instances_solved} shows this information.

\begin{table}[ht]
	\centering
	\begin{tabular}{ccccccc}
		\cline{3-7}
		&  & \texttt{CPLEX} & \multicolumn{2}{c}{\texttt{Auto\_BD}} & \multicolumn{2}{c}{\texttt{BD\_CW}} \\ \hline
		$\alpha$  & $\lambda$ & \texttt{t} &    \texttt{t}       &  \texttt{cuts}        &     \texttt{t}      &     \texttt{cuts}     \\ \hline
		\multirow{5}{*}{0.25} & 0 & - &     -      &     -     &    -       &    -      \\
		& 0.25 & 650.11 &     593.75   &   9423     &   399.88    &  17910     \\
		& 0.5 & 565.71 &   376.23   &     8101   &     436.89   &   16963    \\
		& 0.75 & 1784.65 &     1391.18    &    8747  &   395.65   &    15774    \\
		& 1 & 304.72 &    139.57   &      266   &    178.77    &    13038   \\ \hline
		\multirow{5}{*}{0.4} & 0 & 781.80 &    553.34 &   6190    &  2463.12     &    16697      \\
		& 0.25 & 2658.41 &   703.57  &     9703   &   1464.16  &   18700     \\
		& 0.5 & 1960.28 &     904.24  &  17592  &    1186.19  &     25258     \\
		& 0.75 & 1824.15 &   1240.95 &   11272    &     915.15  &   17235    \\
		& 1 & 1392.97 &     147.06    &   505  &    751.05  &   18882   \\ \hline
	\end{tabular}
	\caption{Comparing the performance of the three algorithms for (CD) considering those instances in block \texttt{A}.}
	\label{table:comp_results_1}
\end{table}

By observing Tables \ref{table:comp_results_1}, \ref{table:comp_results_2} and \ref{table:comp_results_3} in this section, we have observed similar conclusions for the two values considered for $\alpha$.

\begin{itemize}
	\item If $\lambda=0$, our Branch-and-Benders cut approach that generates facet-defining cuts is not competitive. The direct use of \texttt{CPLEX} is the best option, except for those instances that belong to block \texttt{B} and being $\alpha=0.25$. In this case, the \texttt{Auto\_BD} is the best option.
	\item If $\lambda=1$, \texttt{Auto\_BD} is the most competitive. In this case, all the instances were solved to optimality. It seems that solving the center problem takes less time than for the median problem.
	\item If $\lambda=0.25$ and $\alpha=0.25$, our Branch-and-Benders cut approach \texttt{BD\_CW} is the most competitive for any of the blocks of instances considered. In this case, the resolution times are almost at least 200 seconds shorter and \texttt{gaps} are at least 1.9\% smaller. Nevertheless, if $\lambda=0.25$ and $\alpha=0.4$, \texttt{auto\_BD} is the best option in blocks \texttt{A} and \texttt{C}.
\end{itemize}

\begin{itemize}
	\item If $\lambda=0.5$, our \texttt{BD\_CW} is the most competitive for blocks \texttt{B} and \texttt{C}. That is, for those instances in \texttt{B} its associated \texttt{gaps} are at least 4.6\% percent smaller. Besides, all of the instances in \texttt{C} were solved to optimality using \texttt{BD\_CW}, but not all using \texttt{Auto\_BD}. \texttt{Auto\_BD} is the best option for those instances in block \texttt{A}, which are a minority.
	\item If $\lambda=0.75$, our Branch-and-Benders cut approach \texttt{BD\_CW} is the most competitive for any of the blocks of instances considered. In block \texttt{A}, the resolution times are 315 seconds shorter. In \texttt{B}, \texttt{gaps} are at least 4\% smaller. Almost all the instances in \texttt{C} were solved to optimality using \texttt{BD\_CW} and in less time.
\end{itemize}

\begin{table}[htpb]
	\centering
	\begin{tabular}{ccccccc}
		\cline{3-7}
		&  & \texttt{CPLEX} & \multicolumn{2}{c}{\texttt{Auto\_BD}} & \multicolumn{2}{c}{\texttt{BD\_CW}} \\ \hline
		$\alpha$  & $\lambda$ & \texttt{gap} &    \texttt{gap}       &  \texttt{cuts}        &     \texttt{gap}      &     \texttt{cuts}     \\ \hline
		\multirow{5}{*}{0.25} & 0 & 1.94 &     2.72    &   11465   &     3.80    &   18672     \\
		& 0.25 & 5.98 &    3.99  &   20490    &     2.33      &    23217      \\
		& 0.5 & 8.05  &    5.68       &     12026  &    0.98   &     22555   \\
		& 0.75 & - &    -       &   -       &   -        &   -       \\
		& 1 & - &     -      &    -      &    -       &      -    \\ \hline
		\multirow{5}{*}{0.4} & 0 & 1.76 &    3.88     &    11488   &    5.57     &    20765      \\
		& 0.25 & 6.92 &    5.89      &     14723  &     4.85   &    23050    \\
		& 0.5 & 9.40 &     8.54      &    17592 &     3.96   &    25358    \\
		& 0.75 & 11.55 &    12.03  &    19758  &     7.46   &     26672   \\
		& 1 & - &    -    &  -      & -       &     -   \\ \hline
	\end{tabular}
	\caption{Comparing the performance of the three algorithms for (CD) considering those instances in block \texttt{B}.}
	\label{table:comp_results_2}
\end{table}

\begin{table}[htpb]
	\centering
	\begin{tabular}{cccccccccc}
		\cline{3-10}
		&  & \multicolumn{2}{c}{\texttt{CPLEX}} & \multicolumn{3}{c}{\texttt{Auto\_BD}} & \multicolumn{3}{c}{\texttt{BD\_CW}} \\ \hline
		$\alpha$  & $\lambda$ & \texttt{t} &    \texttt{gap}       &  \texttt{t} &  \texttt{gap}  & \texttt{cuts}        &     \texttt{t} & \texttt{gap}     &     \texttt{cuts}     \\ \hline
		\multirow{5}{*}{0.25} & 0 & 1910.51 & 0.21  &   2311.26    &    0.08  &    8986    &  3600   & 2.15 &    18084  \\
		& 0.25 & 2706.31 &   0.60   &  2440.15   &    0.89      &    9665       &  2188.04    & 0.53 &  19471   \\
		& 0.5 & 3037.74 &   4.78     & 2837.76  & 2.88  &  11502   & 934.64 &  0   &   19979   \\
		& 0.75 & 2999.38 &    1.36    & 3145.60  &     3.91     &  10778   &  1042.03     &   0    & 20805   \\
		& 1 & - &  -   &   -   &    -    &  -      &    -  & - &  -   \\ \hline
		\multirow{5}{*}{0.4} & 0 & 1051.13 & 0  & 3199.12    &   0.99 &  8132  &  3600  & 3.19 &  17395  \\
		& 0.25 & 3483.18  &  4.02  &   1546.28    &     0     &    10234    &  2462.57 & 0.36 &     20010   \\
		& 0.5 & 3600 &   7.36   &  1822.76   &   1.38   &     11915  &    2254.52  & 0 & 23132  \\
		& 0.75 & 3600 &   7.67   &  2324.73   &   1.84   &      12707     &   1882.99   & 0.13 & 21696  \\
		& 1 & - &   -   &  -   &  -   &      -     &   -    & - &  -  \\ \hline
	\end{tabular}
	\caption{Comparing the performance of the three algorithms for (CD) considering those instances in block \texttt{C}.}
	\label{table:comp_results_3}
\end{table}

\section{Conclusions}\label{sec:conclusions} 

In this paper, we studied the $\lambda$-cent-dian and the generalized-center problems in Network Design. Both problems aim to minimize a linear combination of the maximum and average traveled distances. The $\lambda$-cent-dian problem, where $\lambda\in [0,1]$, minimizes a convex combination of these objectives, while the generalized-center problem minimizes the difference between them. We explored these concepts under two versions of Pareto-optimality: the first version considers the shortest paths of each origin/destination (O/D) pair, while the second version addresses both objective functions simultaneously. Regarding the second version, we found that the introduction of the new concept of maximum $\lambda$-cent-dian is necessary to generate the entire set of Pareto-optimal solutions.

Furthermore, we proposed mathematical formulations for the first time for these problems.

The solutions for the generalized-center problem can often be deemed inefficient, as they tend to inflate the median value artificially. To address this, we introduce an efficiency constraint to the problem, ensuring that the median value does not deviate significantly from the median value observed in the median network. This modified approach to the generalized-center problem can be likened to performing a lexicographic optimization of both the median and generalized-center objectives.

Additionally, we illustrated the $\lambda$-cent-dian, where $\lambda\in[0,1]$, and the generalized-center solutions using various inequality measures. In scenarios where the efficiency constraint is considered, we have confirmed that the generalized-center solution is not consistently dominated.

Finally, given the hardness of these problems, for case $\lambda\in[0,1]$, we have studied and formulated a branch-and-Benders-cut method. Our computational results show that our method for (CD) is competitive against the one proposed by \texttt{CPLEX} for medium-sized instances with 40 nodes.

\section{Compliance with Ethical Standards}

This work is partially supported by Ministerio de Ciencia e Innovaci\'on under grant PID2020-114594GB-C21 funded by MICIU/AEI/10.13039/501100011033, grant US-1381656 funded by Programa Operativo FEDER/Andalucía and grant ANID PIA AFB230002 funded by the Instituto Sistemas Complejos de Ingenier\'ia.

All authors declare that they have no conflict of interest.

This article does not contain studies with human participants or animals performed by any of the authors.

\begin{appendices}
	\section{Proof of Proposition \ref{prop_C3:full_dimensional}}\label{appendix:proof}
	\begin{proof}
		we define the subgraph \( (\tilde{N}^w,\tilde{E}^w) \) of \( N^w \) induced by a feasible path from \( w^s \) to \( w^t \). To prove the result, we exhibit $\vert N\vert + \vert E\vert + \vert W\vert + 2$ affinely independent feasible points:
		\begin{itemize}
			\item $y_i = 0,\, i\in N$, \quad $x_e = 0,\, e\in E$, \quad $\zeta^w=u^w,\, w\in W$, \quad $\gamma^{max}=\max\limits_{w\in W}\{\zeta^w\}$. 
			\item $y_i = 0,\, i\in N$, \quad $x_e = 0,\, e\in E$, \quad $\zeta^w=u^w,\, w\in W$, \quad $\gamma^{max}=2\,\max\limits_{w\in W}\{\zeta^w\}$.
			\item For each $i\in N$, the points:
			\begin{gather*}
				y_{i}=1, \, y_{i'} = 0, \, i' \in N\setminus\{i\},\quad x_e = 0,\, e\in E,\quad \gamma^w =u^w, \, w\in W, \gamma^{max}=\max\limits_{w\in W}\{\zeta^w\}.
			\end{gather*}
			\item For each $e=\{i,j\}\in E$, the point:
			\begin{gather*}
				y_k=1, \, k\in e,\, y_k = 0,\, k \in N \setminus \{i,j\},\quad x_e=1,\, x_{e'} = 0,\, e'\in E\setminus\{e\},\\
				\zeta^w =u^w, \, w\in W, \quad \gamma^{max} = \max\limits_{w\in W}\{\zeta^w\}.
			\end{gather*}
			\item For each $w\in W$, the point:
			\begin{gather*}
				y_i=1,\, i \in \widetilde{N}^w,\, y_i=0,\, i \in N\setminus \widetilde{N}^w, \quad	x_e=1, \, e\in \widetilde{E}^w,\, x_e = 0,\, e\in E\setminus\widetilde{E}^w,\\
				\zeta^w=2\,u^w,\, \zeta^{w'} = u^w, \, w'\in W\setminus\{w\},\quad
				\gamma^{max} = \max\limits_{w\in W}\{\zeta^w\}.
			\end{gather*}
		\end{itemize}
		If we write them as rows of a matrix of dimension \mbox{$(|N|+|E|+|W|+2)\times(|N|+|E|+|W|+1)$}, using elementary row operations we obtain a lower echelon form in which all its columns are pivots.	
	\end{proof}
\end{appendices}

\bibliography{sn-bibliography}

\end{document}